\documentclass[smallextended]{svjour3}

\smartqed

\usepackage{verbatim,a4wide}
\usepackage{amsmath}
\usepackage{amssymb}
\usepackage{amsthm}
\usepackage[a4paper,left=2.5cm,right=2.5cm,top=2cm,bottom=3cm]{geometry}
\usepackage{graphicx}
\usepackage{caption}
\usepackage{subfig}
\usepackage{grffile}
\usepackage{float}
\usepackage{bbm}
\usepackage{indentfirst}
\usepackage{booktabs}
\usepackage{multirow}
\usepackage{rotating}
\usepackage[T1,hyphens]{url}
\usepackage{hyperref}
\usepackage{epstopdf}

\usepackage[cp1250]{inputenc}
\usepackage[OT4]{fontenc}
\usepackage[english]{babel}

\numberwithin{equation}{section}
\numberwithin{theorem}{section}
\numberwithin{problem}{section}
\numberwithin{lemma}{section}
\numberwithin{corollary}{section}
\numberwithin{remark}{section}
\numberwithin{example}{section}
\numberwithin{proposition}{section}

\makeatletter
\let\c@proposition\c@theorem
\let\c@corollary\c@theorem
\let\c@remark\c@theorem
\let\c@problem\c@theorem
\let\c@lemma\c@theorem
\let\c@definition\c@theorem
\let\c@example\c@theorem
\makeatother

\newcommand{\ra}[1]{\renewcommand{\arraystretch}{#1}}

\newcommand{\N}{\mathbb N}

\newcommand{\dx}{\mbox{${\rm\,d}x$}}

\begin{document}


\title{Efficient modified Jacobi-Bernstein basis transformations}

\author{Przemys{\l}aw Gospodarczyk        \and
        Pawe{\l} Wo\'{z}ny
}

\institute{P. Gospodarczyk (Corresponding author) \at Institute of Computer Science, University of Wroc{\l}aw,
              ul.~F.~Joliot-Curie 15, 50-383 Wroc{\l}aw, Poland\\
              Fax: +48713757801\\
              \email pgo@ii.uni.wroc.pl
           \and P. Wo\'{z}ny \at Institute of Computer Science, University of Wroc{\l}aw,
                ul.~F.~Joliot-Curie 15, 50-383 Wroc{\l}aw, Poland\\
                \email Pawel.Wozny@cs.uni.wroc.pl
}

\date{\today}

\maketitle

\begin{abstract}
In the paper, we show that the transformations between modified Jacobi and Bernstein bases of the constrained space of polynomials of degree at most $n$
can be performed with the complexity $O(n^2)$. As a result, the algorithm of degree reduction of B\'ezier curves that was first presented in (Bhrawy et al., J.~Comput.~Appl.~Math.~302 (2016), 369--384), and then corrected in (Lu and Xiang, J.~Comput.~Appl.~Math.~315 (2017), 65--69), can be significantly improved, since the necessary transformations are done in those papers with the complexity $O(n^3)$. The comparison of running times shows that our transformations are also faster in practice.


\end{abstract}


\section{Introduction}\label{sec:intro}

Recently, Bhrawy et al.~\cite{B16} presented an algorithm of degree reduction of B\'ezier curves with parametric continuity constraints.
In \cite{LX16}, Lu and Xiang corrected that algorithm. Moreover, they solved the problem of degree
reduction of B\'ezier curves with geometric continuity constraints by extending the original algorithm.
Both methods of degree reduction are based on transformations between modified Jacobi and Bernstein bases of the
\textit{constrained space of polynomials of degree at most $n$}. Those transformations are computed in \cite{B16,LX16}
with the complexity $O(n^3)$ (see \cite[Theorems 1 and 2]{LX16}).

In \cite{Woz13}, efficient  transformations between shifted Jacobi and Bernstein bases of the
\textit{unconstrained space of polynomials of degree at most $n$} were proposed. The idea was to consider the connection coefficients written in terms of Hahn polynomials. Those representations were given by Ciesielski \cite{Cie87} and Ronveaux et al.~\cite{R98}. Then, recurrence relations for the coefficients were obtained using recurrence relations for Hahn polynomials. Consequently, the transformations in \cite{Woz13},
which can be used in a method of \textit{unconstrained degree reduction of B\'ezier curves}, are performed with the complexity $O(n^2)$.
Unfortunately, it seems that those results have not been noticed by the CAGD community. Perhaps because in CAGD we often look for an optimal element
(e.g., a degree reduced curve) that is constrained by some continuity conditions. As a result, transformations between bases of the constrained
space of polynomials are more useful. Therefore, the main goal of this paper is to recall the approach from \cite{Woz13} and generalize it
in order to give methods of computing the transformations between modified Jacobi and Bernstein bases with the lowest complexity among existing algorithms, namely $O(n^2)$. The new results significantly improve the methods of \textit{constrained degree reduction of B\'ezier curves} from \cite{B16,LX16}.

In \cite{Doh14}, Doha et al.~proposed an algorithm of constrained degree reduction of B\'ezier curves based on the generalized Jacobi-Bernstein basis transformations. However, those transformations are performed there with the complexity $O(n^3)$. Since the generalized Jacobi polynomials are closely related to the shifted Jacobi polynomials, we notice that the searched connection coefficients depend on Hahn polynomials. Consequently, they can be computed with the complexity $O(n^2)$ using recurrence relations similar to the ones that we present for the modified Jacobi-Bernstein basis transformations.

The paper is organized as follows.
In Section \ref{Sec:Prelim}, we recall some useful definitions and properties. Next, we give two different recurrence relations for the coefficients
of the Bernstein form of the modified Jacobi polynomials (see Section \ref{Sec:JtB}). As a result, there are two different ways of computing those coefficients. Both methods have the complexity $O(n^2)$. In Section \ref{Sec:BtJ}, we solve the reverse problem, i.e., we propose two different methods of
computing the coefficients of the modified Jacobi form of the Bernstein polynomials. Once again, our algorithms have the complexity $O(n^2)$ and they are based on recurrence relations. Moreover, a remark on the generalized Jacobi-Bernstein basis transformations is made. In Section~\ref{Sec:Ex}, we compare the running times of our methods with the running times of the methods from \cite{B16,LX16}. Section~\ref{Sec:Conc} concludes the paper.

\section{Preliminaries}\label{Sec:Prelim}

\subsection{Bernstein and shifted Jacobi bases of the unconstrained space of polynomials}\label{Sec:PrelimJB}

Let $\Pi_n$ be the space of polynomials of degree at most $n$. As is known, \textit{Bernstein polynomials} of degree $n$,
\begin{equation}\label{E:Bern}
   B^n_{j}(x):=\binom nj x^j(1-x)^{n-j} \qquad (j=0,1,\ldots,n),
\end{equation}
form a basis of this space. The \textit{shifted Jacobi polynomials},
\begin{equation}\label{E:Jac}
R_i^{(\alpha,\beta)}(x):= \frac{(\alpha+1)_i}{i!}\sum_{j=0}^{i}\frac{(-i)_j(i+\alpha+\beta+1)_j}{j!(\alpha+1)_j}(1-x)^j \qquad
(i=0,1,\ldots,n;\;\alpha,\beta > -1),
\end{equation}
where
$$
(h)_0 := 1, \qquad (h)_i := h(h+1)\cdots(h+i-1),
$$
form a different basis of $\Pi_n$. In contrast to Bernstein polynomials, they are orthogonal
with respect to the \textit{Jacobi inner product}
\begin{equation}\label{E:Inn}
\langle f,\,g \rangle:=\int_{0}^{1} (1-x)^\alpha x^\beta f(x)g(x)\dx \qquad (\alpha,\beta > -1),
\end{equation}
and, as a result, more useful in the context of degree reduction of polynomial curves with respect to the weighted $L_2$-norm.
In \cite{Woz13}, one of us presented efficient methods of conversion between shifted Jacobi and Bernstein bases.
Those algorithms have the complexity $O(n^2)$. See also \cite{CW02,Cie87,Doh14,Far00,LZ98,Rab03,Rab04,R98,Sun05}.

\subsection{Bernstein and modified Jacobi bases of the constrained space of polynomials}\label{Sec:ConPrelimJB}

In order to solve the problem of degree reduction of B\'ezier curves with continuity constraints at the endpoints,
it is useful to consider the following restriction of the space $\Pi_n$ (see, e.g., \cite{GLW15,GLW17,WL09}).
Let $\Pi_n^{(k,l)}$, where $k$ and $l$ are nonnegative integers such that $k+l\le n$,
be the space of all polynomials of  degree at most $n$,
whose derivatives of order less than $k$ at $t=0$, as well as derivatives of order
less than $l$ at $t=1$, vanish:
\[
\Pi_n^{(k,l)}:=
       \left\{P\in\Pi_n\::\:
       P^{(i)}(0)=0\quad (0\le i\le k-1)
       \; \mbox{and}\;
       P^{(j)}(1)=0\quad (0\le j\le l-1)\right\}.
\]
 As is known,  $\mbox{dim}\;\Pi_n^{(k,l)}=n-k-l+1$, and the following Bernstein polynomials:
\begin{equation}\label{E:Bern2}
B^n_k,B^n_{k+1},\ldots,B^n_{n-l}
\end{equation}
(cf.~\eqref{E:Bern}) form a basis of this space. Observe that $\Pi_n^{(0,0)} \equiv \Pi_n$.
Recall that the constrained problem of degree reduction of B\'ezier curves is often formulated
as a minimization problem of the weighted $L_2$-distance (see, e.g., \cite{HB16,B16,Doh14,GLW15,LX16,WL09}).
Therefore, an orthogonal basis of the space $\Pi_n^{(k,l)}$ with respect to the inner product \eqref{E:Inn} can play a crucial
role in that context. Since Bernstein polynomials \eqref{E:Bern2} are not orthogonal, a different basis is needed.
In \cite{B16}, Bhrawy et al.~introduced the \textit{modified Jacobi polynomials},
\begin{equation}\label{E:MJP}
J_{i,k,l}^{(\alpha,\beta)}(x) := (1-x)^lx^kR^{(\alpha+2l,\beta+2k)}_{i-k-l}(x) \qquad
(i=k+l,k+l+1,\ldots,n;\;\alpha,\beta > -1)
\end{equation}
(cf.~\eqref{E:Jac}), which form an orthogonal basis of the space $\Pi_n^{(k,l)}$ with respect to the inner product
\eqref{E:Inn}. In this paper, we generalize the results from \cite{Woz13} in order to show that the transformations
between the modified Jacobi \eqref{E:MJP} and Bernstein bases \eqref{E:Bern2} can be done with the complexity $O(n^2)$.
More precisely, we give methods of computing the connection coefficients $c_{ih} \equiv c_{ih}(n,k,l;\alpha,\beta)$ and $d_{hi} \equiv d_{hi}(n,k,l;\alpha,\beta)$ that satisfy
\begin{align}
&J_{i,k,l}^{(\alpha,\beta)}(x) = \sum_{h=k}^{n-l} c_{ih}B_h^n(x) \qquad (i = k+l,k+l+1,\ldots,n),\notag\\[1ex]
&B_h^n(x) = \sum_{i=k+l}^{n} d_{hi}J_{i,k,l}^{(\alpha,\beta)}(x) \qquad (h = k,k+1,\ldots,n-l).\label{E:BdJDef}
\end{align}
Recall that these coefficients were first given in \cite{B16}, and then corrected
in \cite{LX16}, namely
\begin{align}
&c_{ih} = \binom{n}{h}^{-1}\sum_{r=\max(0,h+i-n-k)}^{\min(h-k,i-l-k)}(-1)^{i-l-k-r}\binom{i+\alpha+l-k}{r}\binom{i+\beta-l+k}{i-l-k-r}\binom{n-i}{h-k-r},\label{E:cLX}\\[1ex]
&d_{hi} =  g_{i,l,k}^{(\alpha,\beta)}\sum_{r=0}^{i-l-k}(-1)^{i-l-k-r}
\binom{i+\alpha+l-k}{r}\binom{i+\beta-l+k}{i-l-k-r}\binom{n+i+\alpha+\beta}{h+\beta+k+r}^{-1},\label{E:dLX}
\end{align}
where
$$
g_{i,l,k}^{(\alpha,\beta)} := \binom{n}{h}\frac{(2i+\alpha+\beta+1)(i-l-k)!\Gamma(i+l+k+\alpha+\beta+1)}{(n+i+\alpha+\beta+1)\Gamma(i+l-k+\alpha+1)\Gamma(i-l+k+\beta+1)},
$$
and the binomial coefficients are generalized to noninteger arguments,
$$
\binom{y}{t} := \frac{\Gamma(y+1)}{\Gamma(t+1)\Gamma(y-t+1)},
$$
with $\Gamma$ being the gamma function (see, e.g., \cite[Section 1.1]{AAR99}). However, such an approach leads to the complexity $O(n^3)$. Moreover, cumbersome computations of the gamma functions are required. As we shall see, our methods are not only more efficient but also avoid computing the gamma functions.

\subsection{Hahn and dual Hahn polynomials}\label{Sec:PrelimHahn}

Now, we give a short introduction to Hahn and dual Hahn polynomials because they are main tools in our
efficient methods of modified Jacobi-Bernstein basis transformations.

\textit{Hahn polynomials} are given by
\begin{equation}\label{E:H}
Q_n(x;\, \alpha,\, \beta,\, N) := \sum_{j=0}^{n}\frac{(-n)_j(n+\alpha+\beta+1)_j(-x)_j}{j!(\alpha+1)_j(-N)_j}
\qquad (n=0,1,\ldots,N;\; N \in \N),
\end{equation}
where $\alpha, \beta > -1$. They are orthogonal with respect to a discrete inner product (see \cite[(1.5.2)]{KS98}), and satisfy a
\textit{three-term recurrence relation} (see, e.g., \cite[(1.5.3)]{KS98}),
\begin{equation}\label{E:HRec}
-xQ_n(x) = A_nQ_{n+1}(x) - (A_n + C_n)Q_n(x) + C_nQ_{n-1}(x),
\end{equation}
where $Q_n(x) \equiv Q_n(x;\, \alpha,\, \beta,\, N)$,
\begin{align*}
& A_n := \frac{(n+\alpha+\beta+1)(n+\alpha+1)(N-n)}{(2n+\alpha+\beta+1)(2n+\alpha+\beta+2)},\\[1ex]
& C_n := \frac{n(n+\alpha+\beta+N+1)(n+\beta)}{(2n+\alpha+\beta)(2n+\alpha+\beta+1)}.
\end{align*}
Further on in the paper, we will need the following \textit{symmetry property} (see, e.g., \cite[p.~346]{AAR99}):
\begin{equation}\label{E:QSym}
Q_n(x;\, \alpha,\, \beta,\, N)=(-1)^n \frac{(\beta+1)_n}{(\alpha+1)_n} Q_n(N-x;\, \beta,\, \alpha,\, N).
\end{equation}

\textit{Dual Hahn polynomials} are given by
\begin{equation}\label{E:Rn}
R_n(\lambda(x);\, \alpha,\, \beta,\, N) := \sum_{j=0}^{n}\frac{(-n)_j(x+\alpha+\beta+1)_j(-x)_j}{j!(\alpha+1)_j(-N)_j}
\qquad (n=0,1,\ldots,N;\; N \in \N),
\end{equation}
where $\alpha, \beta > -1$ and $\lambda(x) = x(x+\alpha+\beta+1)$. They are orthogonal with respect to a discrete inner product (see \cite[(1.6.2)]{KS98}), and satisfy a three-term recurrence relation (see, e.g., \cite[(1.6.3)]{KS98}),
\begin{equation}\label{E:DRec}
\lambda(x)R_n(\lambda(x)) = B_nR_{n+1}(\lambda(x)) - (B_n + D_n)R_n(\lambda(x)) + D_nR_{n-1}(\lambda(x)),
\end{equation}
where $R_n(\lambda(x)) \equiv R_n(\lambda(x);\, \alpha,\, \beta,\, N)$,
\begin{align*}
& B_n := (n+\alpha+1)(n-N),\\[1ex]
& D_n := n(n-\beta-N-1).
\end{align*}
As is known, Hahn and dual Hahn polynomials are related (see, e.g., \cite[Section 1.6]{KS98}),
\begin{equation}\label{dHH}
R_n(\lambda(x);\, \alpha,\, \beta,\, N) = Q_x(n;\, \alpha,\, \beta,\, N).
\end{equation}

\section{Bernstein form of modified Jacobi polynomials}\label{Sec:JtB}

The following lemma is a generalization of the result from \cite{Cie87}, where the \textit{classic Jacobi polynomials} on the interval $[-1,\,1]$ were considered, and from \cite{Woz13}, where the shifted Jacobi polynomials \eqref{E:Jac} were studied which corresponds to the case without any constraints, i.e., $k=l=0$ (see \eqref{E:MJP}).

\begin{lemma}\label{L:JB}
Modified Jacobi polynomials \eqref{E:MJP} have the following representation in the Bernstein basis \eqref{E:Bern2}:
\begin{equation*}
J_{i,k,l}^{(\alpha,\beta)}(x) = \sum_{h=k}^{n-l} c_{ih}B_h^n(x) \qquad (i = k+l,k+l+1,\ldots,n),
\end{equation*}
where
\begin{equation}\label{E:JBHc}
c_{ih} := \frac{(\alpha+2l+1)_{i-k-l}}{(i-k-l)!}\binom{n}{h}^{-1}\binom{n-k-l}{h-k}Q_{i-k-l}(n-l-h;\, \alpha+2l,\, \beta+2k,\, n-k-l).
\end{equation}
\end{lemma}
\begin{proof}
According to \cite[Theorem 3.1]{Woz13}, shifted Jacobi polynomials \eqref{E:Jac} have the following Bernstein form:
\begin{equation}\label{E:RBrep0}
R_{i}^{(\alpha,\beta)}(x)  = \sum_{h=0}^{n}a_{ih}^{(n,\alpha,\beta)}B_{h}^{n}(x) \qquad (i=0,1,\ldots,n),
\end{equation}
where
\begin{equation}\label{E:aih}
a_{ih}^{(n,\alpha,\beta)} := (-1)^{i}\frac{(\beta+1)_i}{i!}Q_{i}(h;\,\beta,\,\alpha,\,n).
\end{equation}
Notice that the use of the symmetry property \eqref{E:QSym} in \eqref{E:aih} results in
$$
a_{ih}^{(n,\alpha,\beta)} = \frac{(\alpha+1)_i}{i!}Q_{i}(n-h;\,\alpha,\,\beta,\,n).
$$
Now, suitable substitutions and indices manipulations in \eqref{E:RBrep0} give us
\begin{equation}\label{E:RBrep}
R_{i-k-l}^{(\alpha+2l,\beta+2k)}(x)  = \sum_{h=k}^{n-l}a_{i-k-l,h-k}^{(n-k-l,\alpha+2l,\beta+2k)}B_{h-k}^{n-k-l}(x) \qquad (i=k+l,k+l+1,\ldots,n).
\end{equation}
Then, we multiply both sides of the equation \eqref{E:RBrep} by $(1-x)^lx^k$, use \eqref{E:MJP}, and after some algebra, we get
\begin{equation*}
J_{i,k,l}^{(\alpha,\beta)}(x)  = \sum_{h=k}^{n-l}\binom{n}{h}^{-1}\binom{n-k-l}{h-k}a_{i-k-l,h-k}^{(n-k-l,\alpha+2l,\beta+2k)}B_{h}^{n}(x) \qquad (i=k+l,k+l+1,\ldots,n),
\end{equation*}
which completes the proof.
\end{proof}

In Theorems \ref{T:T1} and \ref{T:T2}, we give two different recurrence relations for the connection coefficients $c_{ih}$ $(i=k+l,k+l+1,\ldots,n;\;h=k,k+1,\ldots,n-l)$. As a result, there are two different methods of computing those coefficients.
Notice that both methods have the complexity $O(n^2)$. Recall that a similar recurrence relation for the unconstrained case
$k=l=0$ was given in \cite[Lemma 4.1]{Woz13}.

\begin{theorem}\label{T:T1}
For a fixed $i$, the connection coefficients $c_{ih}$ $(h=k,k+1,\ldots,n-l)$ given by \eqref{E:JBHc} satisfy the following recurrence relation:
\begin{align}
&c_{i,n-l} = \frac{(\alpha+2l+1)_{i-k-l}}{(i-k-l)!}\binom{n}{l}^{-1},\label{E:c1}\\
&c_{i,n-l-1} = c_{i,n-l}\frac{(n-k-l)(l+1)}{n-l}\left[1-\frac{(k+l-i)(i+k+l+\sigma)}{(k+l-n)(\alpha+2l+1)}\right],\label{E:c2}\\
&c_{ih} = F_{i}(h)c_{i,h+1} + G(h)c_{i,h+2} \qquad (h = n-l-2,n-l-3,\ldots,k),\label{E:c3}
\end{align}
where
\begin{align}
&\sigma := \alpha+\beta+1,\label{E:sigma}\\
&F_{i}(h) := \frac{(n-h)(h+1-k)}{(n-l-h)(h+1)}\left[1-H(h)-\frac{(k+l-i)(i+k+l+\sigma)}{(n+l+\alpha-h)(k-h-1)}\right],\notag\\
&G(h) := \frac{(n-h-1)_2(h+1-k)_2}{(n-l-h-1)_2(h+1)_2}H(h)\notag
\end{align}
with
$$
H(h) := \frac{(n-l-h-1)(h+k+\beta+2)}{(n+l+\alpha-h)(k-h-1)}.
$$

\end{theorem}
\begin{proof}
According to Lemma \ref{L:JB}, the coefficients $c_{ih}$ can be represented using Hahn polynomials. Since Hahn and dual Hahn polynomials are related (see \eqref{dHH}), we obtain
\begin{equation}\label{E:JBdHc}
c_{ih} = \frac{(\alpha+2l+1)_{i-k-l}}{(i-k-l)!}\binom{n}{h}^{-1}\binom{n-k-l}{h-k}R_{n-l-h}(\lambda(i-k-l);\, \alpha+2l,\, \beta+2k,\, n-k-l)
\end{equation}
(cf.~\eqref{E:JBHc}). Then, it can be checked that \eqref{E:c1} and \eqref{E:c2} follow from \eqref{E:JBdHc} for $h=n-l, n-l-1$, respectively (see \eqref{E:Rn}). Finally, the application of the recurrence relation \eqref{E:DRec} to \eqref{E:JBdHc}, along with some algebra, gives \eqref{E:c3}.
\end{proof}

\begin{theorem}\label{T:T2}
For a fixed $h$, the connection coefficients $c_{ih}$ $(i=k+l,k+l+1,\ldots,n)$ given by \eqref{E:JBHc} satisfy the following recurrence relation:
\begin{align}
&c_{k+l,h} = \binom{n}{h}^{-1}\binom{n-k-l}{h-k},\label{E:c4}\\
&c_{k+l+1,h} = c_{k+l,h}\left[\alpha+2l+1-\frac{(\sigma+2k+2l+1)(l+h-n)}{k+l-n}\right],\label{E:c5}\\
&c_{ih} = K_{h}(i)c_{i-1,h} + L(i)c_{i-2,h} \qquad (i=k+l+2,k+l+3,\ldots,n),\label{E:c6}
\end{align}
where $\sigma$ is given by \eqref{E:sigma},
\begin{align*}
&K_{h}(i) := \frac{\alpha+l+i-k}{i-k-l}\left[1-M(i)-\frac{(l+h-n)(2i+\alpha+\beta-1)_2}{(i+k+l+\alpha+\beta)(\alpha+l+i-k)(i-n-1)}\right],\\
&L(i) := \frac{(\alpha+l+i-k-1)_2}{(i-k-l-1)_2}M(i)
\end{align*}
with
$$
M(i) := \frac{(i-k-l-1)(n+i+\alpha+\beta)(i+k+\beta-l-1)(2i+\alpha+\beta)}{(2i+\alpha+\beta-2)(i+k+l+\alpha+\beta)(i+l+\alpha-k)(i-n-1)}
$$
(cf.~Theorem \ref{T:T1}).
\end{theorem}
\begin{proof}
Recall that the coefficients $c_{ih}$ can be represented using Hahn polynomials (see Lemma \ref{L:JB}). Consequently,
\eqref{E:c4} and \eqref{E:c5} are obtained from \eqref{E:JBHc} by setting $i=k+l, k+l+1$, respectively (see \eqref{E:H}). Finally,
we apply the recurrence relation \eqref{E:HRec} to \eqref{E:JBHc}, and then some manipulations lead us to \eqref{E:c6}.
\end{proof}

\section{Modified Jacobi form of Bernstein polynomials}\label{Sec:BtJ}

First of all, we notice that the connection coefficients in \eqref{E:BdJDef} can be represented using Hahn polynomials \eqref{E:H}. The following lemma is a generalization of the result from \cite{R98} (see also \cite{Woz13}), where only the unconstrained case $k=l=0$
(see \eqref{E:MJP}) was considered. More precisely, Bernstein polynomials \eqref{E:Bern} were represented in the \textit{monic} shifted Jacobi basis (cf.~\eqref{E:Jac}).

\begin{lemma}\label{L:BJ}
Bernstein polynomials \eqref{E:Bern2} have the following representation in the modified Jacobi basis \eqref{E:MJP}:
\begin{equation}\label{E:BJRep}
B_h^n(x) = \sum_{i=k+l}^{n} d_{hi}J_{i,k,l}^{(\alpha,\beta)}(x) \qquad (h = k,k+1,\ldots,n-l),
\end{equation}
where
\begin{equation}\label{E:BJd}
d_{hi} := z_{hi}w_{hi}
\end{equation}
with
\begin{align}
&z_{hi} := \binom{n}{h}\frac{(2i+\sigma)(k+l-n)_{i-k-l}(\alpha+2l+1)_{n-l-h}(\beta+2k+1)_{h-k}}{(\alpha+2l+1)_{i-k-l}(i+k+l+\sigma)_{n+1-k-l}},\label{E:mu}\\[1ex]
&w_{hi} := Q_{i-k-l}(h-k;\, \beta+2k,\, \alpha+2l,\, n-k-l),\label{E:q}
\end{align}
and $\sigma$ as defined in \eqref{E:sigma}.
\end{lemma}
\begin{proof}
First, we recall the shifted Jacobi form of Bernstein polynomials \eqref{E:Bern} (see \cite[Theorem 3.2]{Woz13}),
\begin{equation}\label{E:BR0}
B_{h}^{n}(x) = \sum_{i=0}^{n}b_{hi}^{(n,\alpha,\beta)}R_{i}^{(\alpha,\beta)}(x) \qquad (h=0,1,\ldots,n),
\end{equation}
where
\begin{equation}\label{E:bhi}
b_{hi}^{(n,\alpha,\beta)} := \binom{n}{h}\frac{(2i+\sigma)(-n)_i(\alpha+1)_{n-h}(\beta+1)_h}{(\alpha+1)_i(i+\sigma)_{n+1}}Q_{i}(h;\,\beta,\,\alpha,\,n).
\end{equation}
Then, appropriate substitutions and indices manipulations in \eqref{E:BR0} lead us to
\begin{equation}\label{E:BRpf}
B_{h-k}^{n-k-l}(x) = \sum_{i=k+l}^{n}b_{h-k,i-k-l}^{(n-k-l,\alpha+2l,\beta+2k)}R_{i-k-l}^{(\alpha+2l,\beta+2k)}(x) \qquad (h=k,k+1,\ldots,n-l).
\end{equation}
Next, we multiply both sides of the equation \eqref{E:BRpf} by
$\binom{n}{h}\binom{n-k-l}{h-k}^{-1}(1-x)^lx^k$,
\begin{align}\label{E:BRpf2}
B_{h}^{n}(x) = &\sum_{i=k+l}^{n}\binom{n}{h}\binom{n-k-l}{h-k}^{-1}b_{h-k,i-k-l}^{(n-k-l,\alpha+2l,\beta+2k)}(1-x)^lx^k\nonumber\\
&\times R_{i-k-l}^{(\alpha+2l,\beta+2k)}(x) \qquad (h=k,k+1,\ldots,n-l),
\end{align}
and finally obtain the equations \eqref{E:BJRep}--\eqref{E:q} by substituting \eqref{E:MJP} into \eqref{E:BRpf2}.
\end{proof}

\begin{remark}
Notice that the formula \eqref{E:BJd} relating the coefficients $d_{hi}$ with Hahn polynomials is a bit more complicated than the analogous formula from the previous section (cf.~\eqref{E:JBHc}). Therefore, our goal is to first compute separately the quantities $z_{hi}$ and $w_{hi}$ using recurrence relations, and then to obtain $d_{hi}$ using \eqref{E:BJd}.
\end{remark}

In Theorems \ref{T:T3} and \ref{T:T4}, we give two different recurrence relations for each quantities $z_{hi}$ and $w_{hi}$.
Consequently, there are two different methods of computing the connection coefficients $d_{hi}$ $(h = k,k+1,\ldots,n-l;\;i=k+l,k+l+1,\ldots,n)$.
Observe that both methods have the complexity $O(n^2)$. Recall that the unconstrained case $k=l=0$ was solved in \cite[Lemma 4.2]{Woz13} using a similar
approach.

\begin{theorem}\label{T:T3}
For a fixed $h$, the quantities $z_{hi}$ $(i=k+l,k+l+1,\ldots,n)$ given by \eqref{E:mu} satisfy the following recurrence relation:
\begin{align}
&z_{h,k+l} = \binom{n}{h}\frac{(\alpha+2l+1)_{n-l-h}(\beta+2k+1)_{h-k}}{(2k+2l+\sigma+1)_{n-k-l}},\label{E:mu1}\\[1ex]
&z_{hi} = \frac{(2i+\sigma)(i+k+l+\alpha+\beta)(i-n-1)}{(\alpha+l+i-k)(i+n+\sigma)(2i+\alpha+\beta-1)}z_{h,i-1}
\qquad (i = k+l+1,k+l+2,\ldots,n);\label{E:mu2}
\end{align}
and the quantities $w_{hi}$ $(i=k+l,k+l+1,\ldots,n)$ given by \eqref{E:q} satisfy
\begin{align}
&w_{h,k+l} = 1, \quad w_{h,k+l+1} = 1+\frac{(h-k)(2k+2l+\sigma+1)}{(k+l-n)(\beta+2k+1)},\label{E:q1}\\[1ex]
&w_{hi}  = P_{h}(i)w_{h,i-1} + S(i)w_{h,i-2}\qquad (i = k+l+2,k+l+3,\ldots,n),\label{E:qrec}
\end{align}
where $\sigma$ is defined by \eqref{E:sigma},
\begin{align}
&P_{h}(i) := 1-S(i)-\frac{(k-h)(2i+\alpha+\beta-1)_2}{(i+k+l+\alpha+\beta)(i+k+\beta-l)(i-n-1)},\notag\\[1ex]
&S(i) := \frac{(i-k-l-1)(i+\beta+\alpha+n)(i+l+\alpha-k-1)(2i+\beta+\alpha)}{(2i+\beta+\alpha-2)(i+k+l+\beta+\alpha)(i+k+\beta-l)(i-n-1)}.\notag
\end{align}
\end{theorem}
\begin{proof}
The formula \eqref{E:mu1} follows from \eqref{E:mu} for $i=k+l$. The relation \eqref{E:mu2} can be easily proved by induction.
By setting $i = k+l, k+l+1$ in \eqref{E:q} (see also \eqref{E:H}), we obtain \eqref{E:q1}.
Finally, the application of the recurrence relation \eqref{E:HRec} to \eqref{E:q}, combined with some algebraic manipulation,
gives \eqref{E:qrec}.
\end{proof}

\begin{theorem}\label{T:T4}
For a fixed $i$, the quantities $z_{hi}$ $(h=k,k+1,\ldots,n-l)$ given by \eqref{E:mu} satisfy the following recurrence relation:
\begin{align}
&z_{ki} = \binom{n}{k}\frac{(2i+\sigma)(\alpha+l+i+1-k)_{n-i}(k+l-n)_{i-k-l}}{(i+k+l+\sigma)_{n-k-l+1}},\label{E:mu3}\\[1ex]
&z_{hi} = \frac{(n+1-h)(\beta+k+h)}{h(\alpha+l+n+1-h)}z_{h-1,i}
\qquad (h=k+1,k+2,\ldots,n-l);\label{E:mu4}
\end{align}
and the quantities $w_{hi}$ $(h=k,k+1,\ldots,n-l)$ given by \eqref{E:q} satisfy
\begin{align}
&w_{ki} = 1, \quad w_{k+1,i} = 1+\frac{(i-k-l)(i+k+l+\sigma)}{(\beta+2k+1)(k+l-n)},\label{E:q2}\\[1ex]
&w_{hi}  = T_{i}(h)w_{h-1,i} + V(h)w_{h-2,i}\qquad (h=k+2,k+3,\ldots,n-l),\label{E:qrec2}
\end{align}
where $\sigma$ is defined by \eqref{E:sigma},
\begin{align}
&T_{i}(h) := 1 - V(h) - \frac{(k+l-i)(i+k+l+\sigma)}{(h+k+\beta)(h+l-n-1)} ,\notag\\[1ex]
&V(h) := \frac{(h-k-1)(l+n+\alpha+2-h)}{(h+k+\beta)(h+l-n-1)}\notag
\end{align}
(cf.~Theorem \ref{T:T3}).
\end{theorem}
\begin{proof}
The formula \eqref{E:mu3} is obtained from \eqref{E:mu} for $h=k$. The relation \eqref{E:mu4} can be easily proved by induction.
Since Hahn and dual Hahn polynomials are related (see \eqref{dHH}), we have
\begin{equation}\label{E:qdualH}
w_{hi} = R_{h-k}(\lambda(i-k-l);\, \beta+2k,\, \alpha+2l,\, n-k-l)
\end{equation}
(cf.~\eqref{E:q}). Now, it can be checked that the formulas \eqref{E:q2} follow from \eqref{E:qdualH} for $h=k,k+1$ (see \eqref{E:Rn}).
The recurrence relation \eqref{E:qrec2} is obtained, after some algebra, from \eqref{E:qdualH} and \eqref{E:DRec}.
\end{proof}

As a bonus, we give a relation between the connection coefficients $c_{ih}$ (see \eqref{E:JBHc}) and $d_{hi}$ (see \eqref{E:BJd}).
This is a generalization of \cite[Remark 3.4]{Woz13}, where only the case of $k=l=0$ was considered.

\begin{proposition}\label{T:T5}
The connection coefficients $c_{ih}$ (see \eqref{E:JBHc}) and $d_{hi}$ (see \eqref{E:BJd}) are related in the following way:
\begin{equation*}
c_{ih} = u_{ih}d_{hi} \qquad (i = k+l,k+l+1,\ldots,n;\;h=k,k+1,\ldots,n-l),
\end{equation*}
where, for a fixed $i$,
\begin{align}
&u_{ik} := (-1)^{i-k-l}\binom{n}{k}^{-2}\frac{(i+k+l+\sigma)_{n+1-k-l}(\beta+2k+1)_{i-k-l}}{(2i+\sigma)(i-k-l)!
(k+l-n)_{i-k-l}(\alpha+l+i+1-k)_{n-i}},\label{E:verth1}\\[1ex]
&u_{ih} :=  -u_{i,h-1}\frac{h^2(l+h-n-1)(n+l+\alpha+1-h)}{(h-n-1)^2(h-k)(h+k+\beta)} \qquad (h=k+1,k+2,\ldots,n-l);\label{E:rho1}
\end{align}
alternatively, for a fixed $h$,
\begin{align}
&u_{k+l,h} := \binom{n-k-l}{h-k}\binom{n}{h}^{-2}\frac{(2k+2l+\sigma+1)_{n-k-l}}{(\alpha+2l+1)_{n-l-h}(\beta+2k+1)_{h-k}},\label{E:verth2}\\[1ex]
&u_{ih} := -u_{i-1,h}\frac{(i+l+\alpha-k)(2i+\alpha+\beta-1)(n+i+\sigma)(i+k+\beta-l)}{(2i+\sigma)(i-k-l)(i-n-1)(i+k+l+\alpha+\beta)}\nonumber\\
& \hspace*{8.75cm} (i = k+l+1,k+l+2,\ldots,n)\label{E:rho2}
\end{align}
with $\sigma$ as defined in \eqref{E:sigma}.
\end{proposition}
\begin{proof}
First, we apply \eqref{E:QSym} to \eqref{E:q}. Now, we are able to compare $c_{ih}$ (see \eqref{E:JBHc}) with $d_{hi}$ (see \eqref{E:BJd}). We notice that these coefficients depend, in a different way, on the Hahn polynomials with the same parameters. It can be checked that
\begin{align}
u_{ih} = &(-1)^{i-k-l}\frac{(\alpha+2l+1)_{i-k-l}(i+k+l+\sigma)_{n+1-k-l}
(\beta+2k+1)_{i-k-l}}{(2i+\sigma)(i-k-l)!(k+l-n)_{i-k-l}(\alpha+2l+1)_{n-l-h}(\beta+2k+1)_{h-k}}\nonumber\\
&\times\binom{n-k-l}{h-k}\binom{n}{h}^{-2} \qquad (i = k+l,k+l+1,\ldots,n;\;h=k,k+1,\ldots,n-l).\label{E:verth3}
\end{align}
Obviously, \eqref{E:verth1} and \eqref{E:verth2} follow from \eqref{E:verth3} for $h=k$ and $i=k+l$, respectively.
Finally, we can easily prove \eqref{E:rho1} and \eqref{E:rho2} using induction.
\end{proof}

\begin{remark}
In \cite{Doh14}, Doha et al.~presented generalized Jacobi-Bernstein basis transformations and used them in their algorithm of constrained degree reduction of B\'ezier curves. However, those transformations are performed there with the complexity $O(n^3)$. Recall that the \emph{generalized Jacobi polynomials} are closely related to the shifted Jacobi polynomials \eqref{E:Jac} in all four cases,
$$
\hat{J}^{(\alpha,\beta)}_i(x) =
\left\{\begin{array}{lr}
R^{(\alpha,\beta)}_{i}(x) &  (\alpha, \beta > -1),\\[1ex]
(1-x)^{-\alpha} R^{(-\alpha,\beta)}_{i+\alpha}(x) &  (\alpha \in \mathbb{Z}^{-}, \beta > -1),\\[1ex]
x^{-\beta} R^{(\alpha,-\beta)}_{i+\beta}(x) & (\alpha > -1, \beta \in \mathbb{Z}^{-}),\\[1ex]
(1-x)^{-\alpha} x^{-\beta} R^{(-\alpha,-\beta)}_{i+\alpha+\beta}(x) & (\alpha, \beta \in \mathbb{Z}^{-}).
\end{array}\right.
$$
As in the proofs of Lemmas \ref{L:JB} and \ref{L:BJ}, it can be shown that the connection coefficients of the generalized Jacobi-Bernstein basis transformations depend on Hahn polynomials in the following way:
\begin{align*}
&\hat{J}^{(\alpha,\beta)}_i(x) = \sum_{h=0}^{n+\alpha} \binom{n}{h}^{-1}\binom{n+\alpha}{h}a^{(n+\alpha,-\alpha,\beta)}_{i+\alpha, h}
B_h^n(x) \quad (i = -\alpha,-\alpha+1,\ldots,n;\;\alpha \in \mathbb{Z}^{-}, \beta > -1),\\[1ex]
&\hat{J}^{(\alpha,\beta)}_i(x) = \sum_{h=-\beta}^{n} \binom{n}{h}^{-1}\binom{n+\beta}{h+\beta}a^{(n+\beta,\alpha,-\beta)}_{i+\beta, h+\beta}
B_h^n(x) \quad (i = -\beta,-\beta+1,\ldots,n;\;\alpha > -1, \beta \in \mathbb{Z}^{-}),\\[1ex]
&\hat{J}^{(\alpha,\beta)}_i(x) = \sum_{h=-\beta}^{n+\alpha} \binom{n}{h}^{-1}\binom{n+\alpha+\beta}{h+\beta}a^{(n+\alpha+\beta,-\alpha,-\beta)}_{i+\alpha+\beta, h+\beta}
B_h^n(x)\\
&\hphantom{\hat{J} = \sum_{h=-\beta}^{n+\alpha} \binom{n}{h}^{-1}\binom{n+\alpha+\beta}{h+\beta}a^{(n+\alpha+\beta,-\alpha,-\beta)}_{i+\alpha+\beta, h+\beta}} \qquad\qquad\qquad\;
(i = -\alpha-\beta,-\alpha-\beta+1,\ldots,n;\;\alpha, \beta \in \mathbb{Z}^{-}),\\[1ex]
&B_h^{n}(x) = \sum_{i=-\alpha}^{n} \binom{n+\alpha}{h}^{-1}
\binom{n}{h} b_{h,i+\alpha}^{(n+\alpha,-\alpha,\beta)}\hat{J}^{(\alpha,\beta)}_i(x) \quad (h=0,1,\ldots,n+\alpha;\;\alpha \in \mathbb{Z}^{-}, \beta > -1),\\[1ex]
&B_h^{n}(x) = \sum_{i=-\beta}^{n} \binom{n+\beta}{h+\beta}^{-1}
\binom{n}{h} b_{h+\beta,i+\beta}^{(n+\beta,\alpha,-\beta)}\hat{J}^{(\alpha,\beta)}_i(x) \quad (h=-\beta,-\beta+1,\ldots,n;\;\alpha > -1, \beta \in \mathbb{Z}^{-}),\\[1ex]
&B_h^{n}(x) = \sum_{i=-\alpha-\beta}^{n} \binom{n+\alpha+\beta}{h+\beta}^{-1}
\binom{n}{h} b_{h+\beta,i+\alpha+\beta}^{(n+\alpha+\beta,-\alpha,-\beta)}\hat{J}^{(\alpha,\beta)}_i(x) \quad (h=-\beta,-\beta+1,\ldots,n+\alpha;\;\alpha, \beta \in \mathbb{Z}^{-}),
\end{align*}
where $a_{ih}^{(n,\alpha,\beta)}$ and $b_{hi}^{(n,\alpha,\beta)}$ are defined by \eqref{E:aih} and \eqref{E:bhi}, respectively.
Consequently, those coefficients can be computed with the complexity $O(n^2)$ using recurrence relations similar to the ones presented in Theorems \ref{T:T1}, \ref{T:T2}, \ref{T:T3} and \ref{T:T4}.
\end{remark}

\section{Examples}\label{Sec:Ex}

In this section, we compare the running times of our methods based on Theorems \ref{T:T1}, \ref{T:T2}, \ref{T:T3} and \ref{T:T4} with the running times of the methods from \cite{B16,LX16} based on equations \eqref{E:cLX} and \eqref{E:dLX}. The results were obtained on a computer with \texttt{Intel Core i5-3337U 1.8GHz} processor and \texttt{8GB} of \texttt{RAM}, using $16$-digit arithmetic. $\mbox{Maple}{\small \texttrademark}13$ worksheet containing programs and tests is available at \url{http://www.ii.uni.wroc.pl/~pgo/papers.html}.

The running times of the algorithms of computing $c_{ih}$ $(i=k+l,k+l+1,\ldots,n;\;h=k,k+1,\ldots,n-l)$ and $d_{hi}$ $(h=k,k+1,\ldots,n-l;\;i=k+l,k+l+1,\ldots,n)$ are given in Tables \ref{Tab:1} and \ref{Tab:2}, respectively.
For each choice of $\alpha$ and $\beta$ (a row in Tables \ref{Tab:1} and \ref{Tab:2}), each algorithm was executed for $n=5,6,\ldots,15$ hundred times (i.e., 1100 times in total). In the tables, we present total running times of the algorithms for each choice. The following choices of $\alpha$ and $\beta$ were considered:
\begin{itemize}\setlength{\itemindent}{0.35cm}
\item[(i)] fixed \textit{natural} choices of $\alpha$ and $\beta$ (see rows 1--5 of Tables \ref{Tab:1} and \ref{Tab:2});
\item[(ii)] 1100 random pairs $(\alpha, \beta) \in [-0.99,\,1.01) \times [-0.99,\,1.01)$ (see row 6 of Tables \ref{Tab:1} and \ref{Tab:2});
\item[(iii)] for each $n$, $\alpha = -0.9,-0.8,\ldots,9$; $\beta = 0.3,0.4,\ldots,10.2$ (see row 7 of Tables \ref{Tab:1} and \ref{Tab:2}).
\end{itemize}
For all tests, we set $k = l = 1$.

\begin{table}[H]
\captionsetup{margin=0pt, font={small}}
\centering
\ra{1.2}
\scalebox{0.99}{
\begin{tabular}{@{}cccccc@{}}
\toprule \multicolumn{2}{c}{Parameters} & \phantom{abc} & \multicolumn{3}{c}{Total running times [s]}
\\ \cmidrule{1-2} \cmidrule{4-6} $\alpha$ & $\beta$ & \phantom{abc} & Theorem \ref{T:T1} &
Theorem \ref{T:T2} & \cite{B16,LX16}
\\ \midrule
$0$ & $0$ && $1.016$ & $0.922$ & $1.953$\\
$0.5$ & $0.5$ && $1.016$ & $0.922$ & $2.234$\\
$-0.5$ & $-0.5$ && $1.015$ & $0.938$ & $2.047$\\
$-0.5$ & $0.5$ && $1.031$ & $0.906$ & $2.047$\\
$0.5$ & $-0.5$ && $1$ & $0.922$ & $2.094$\\
 \midrule
 random in $[-0.99,\,1.01)$  & random in $[-0.99,\,1.01)$ && $1.25$ & $1.078$ & $22.813$\\
  \midrule
 $-0.9,-0.8,\ldots,9$  & $0.3,0.4,\ldots,10.2$  && $1.125$ & $1$ & $3.937$
\end{tabular}}
\caption{Total running times of the algorithms of computing $c_{ih}$ $(i=k+l,k+l+1,\ldots,n;\;h=k,k+1,\ldots,n-l)$ for $n=5,6,\ldots,15$ hundred times with
different strategies of choosing $\alpha$ and $\beta$. For all tests, $k=l=1$.}
\label{Tab:1}
\end{table}

\newpage

\begin{table}[H]
\captionsetup{margin=0pt, font={small}}
\centering
\ra{1.2}
\scalebox{0.99}{
\begin{tabular}{@{}cccccc@{}}
\toprule \multicolumn{2}{c}{Parameters} & \phantom{abc} & \multicolumn{3}{c}{Total running times [s]}
\\ \cmidrule{1-2} \cmidrule{4-6} $\alpha$ & $\beta$ & \phantom{abc} & Theorem \ref{T:T3} &
Theorem \ref{T:T4} & \cite{B16,LX16}
\\ \midrule
$0$ & $0$ && $1.813$ & $1.437$ & $3.75$\\
$0.5$ & $0.5$ && $2.031$ & $1.469$ & $4.25$\\
$-0.5$ & $-0.5$ && $1.843$ & $1.422$ & $4.063$\\
$-0.5$ & $0.5$ && $1.859$ & $1.438$ & $3.984$\\
$0.5$ & $-0.5$ && $1.859$ & $1.453$ & $3.985$\\
 \midrule
 random in $[-0.99,\,1.01)$  & random in $[-0.99,\,1.01)$ && $2.25$ & $1.718$ & $72.813$\\
  \midrule
 $-0.9,-0.8,\ldots,9$  & $0.3,0.4,\ldots,10.2$  && $2$ & $1.594$ & $20.906$
\end{tabular}}
\caption{Total running times of the algorithms of computing $d_{hi}$ $(h=k,k+1,\ldots,n-l;\;i=k+l,k+l+1,\ldots,n)$ for $n=5,6,\ldots,15$ hundred times with
different strategies of choosing $\alpha$ and $\beta$. For all tests, $k=l=1$.}
\label{Tab:2}
\end{table}

Clearly, our methods are significantly faster than the methods from \cite{B16,LX16} in all considered cases.
Observe that the differences are very large in row 6 of Table \ref{Tab:1}, and rows 6, 7 of Table \ref{Tab:2}.
This is not only because of the difference in computational complexity but also because of some cumbersome
computations of the gamma functions that are required by the methods from \cite{B16,LX16} (see \eqref{E:cLX} and \eqref{E:dLX}).

\section{Conclusions}\label{Sec:Conc}

In the paper, we present efficient transformations between modified Jacobi and Bernstein bases of
the constrained space of polynomials of degree at most $n$. We notice that the searched connection
coefficients can be written in terms of Hahn and dual Hahn polynomials, which results in fast methods of
computing them using recurrence relations. The idea is a generalization of the one from \cite{Woz13},
where only the unconstrained case of the problem was solved. Our new methods have the complexity $O(n^2)$, whereas the complexity of other existing
algorithms is $O(n^3)$ (see \cite{B16,LX16}). Moreover, the comparison of running times shows that our methods are also faster in practice.
Consequently, the methods of constrained degree reduction of B\'ezier curves from \cite{B16,LX16} can be significantly improved.

\bibliographystyle{plain}


\end{document}